\documentclass[reqno]{amsart}
\usepackage{amsmath, amsthm, amssymb}
\usepackage{mathrsfs} 
\usepackage{enumitem}                 
\usepackage{hyperref}
\usepackage{graphicx}

 \newtheorem{theorem}{Theorem}%[section]

\newtheorem{proposition}{Proposition}

\newtheorem{remark}{Remark}

\theoremstyle{definition} % Roman font in body

\def\N{\ensuremath{\mathbb{N}}}
\def\W{\ensuremath{\mathscr{W}}}
\def\E{\ensuremath{\mathscr{E}}}

\def\F{\ensuremath{\mathscr{F}}}
\def\A{\ensuremath{\mathcal{A}}}
\def\G{\ensuremath{\mathcal{G}}}
\def\Zm{\ensuremath{{\mathbb{Z}_-}}}
\def\Ke{\ensuremath{K_e}}

\begin{document}

\title{On $g$-functions for countable state  subshifts}

\author{Adam Jonsson}
\thanks{I am grateful to Wolfgang Krieger for a discussion about subshift presentations.}
 \address{Department of Engineering Sciences and Mathematics \\Lule{\aa} University of Technology, 97187 Lule\aa, Sweden}\email{adam.jonsson@ltu.se} 
 \date{\today}
 \subjclass[2010]{Primary 37B10} 

 \keywords{$g$-functions; subshifts; invariant sets}
\maketitle

\begin{abstract}
This note revisits the problem of finding necessary and sufficient conditions for a subshift to have a continuous $g$-function. Results obtained by Krieger (IMS Lecture Notes-Monograph Series, 48, 306--316, 2006) on finite alphabet subshifts are generalized to countable state subshifts. \end{abstract}

\section{Introduction}\label{sec: intro}
Let $\mathcal{A}$ be a finite or countably infinite set provided with the discrete topology, let $\mathbb{Z}_-=\{0,-1,-2,\dots\}$,  and let $\Theta$ be the shift on $X^{}=\mathcal{A}^{\mathbb{Z}_-}$.  Let $\G$ denote the set of continuous   $g$-functions on $X$.\footnote{The notions of  $g$-function and $g$-measure have their origin in Doeblin and Fortet's \cite{DF37} work on chains of infinite order.  Keane \cite{Kea72}  introduced $g$-functions in ergodic theory. For details and further   references, see, for example, Johansson et al \cite{JOP07} and Stenflo \cite{Ste02, Ste03}.}  That is,  $\mathcal{G}$ is the set of all continuous    $g\colon X\to [0,1]$  with the property that 
\begin{align}\label{eq: sum is one}
\sum_{x'\in \Theta^{-1}(x)}g(x')=1 \text{ for every $x =(\dots{},x_{-1},x_0) \in X^{}$.}
\end{align}
For a subset  $K \subset X^{}$  and  $g \in \mathcal{G}$, we say that  $K$ is $g$-invariant or 
invariant for $g$ if
\begin{equation}\label{def: invariance g}
g(x)=0 \text{ for all $x\in K^c$ such that $\Theta(x) \in K$.}
\end{equation}
 
We consider the following problem:  given a   subset $K \subset X$, are there  $g \in \mathcal{G}$ for which $K$ is invariant?   Gundy \cite[p. 79]{Gun07} raises the special case of this question when $\A=\{0,1\}$ and $K$ is a subshift, that is, a closed   $\Theta$-invariant subset.  Krieger \cite{Kri06} considers the case when $\A$ is finite and gives a   necessary and sufficient condition  for subshifts of the two-sided space   $\mathcal{A}^{\mathbb{Z}}$.  As in \cite{Kri06}, we refer to this condition    as  
\emph{property $G$}. For a subshift  of the one-sided  space, property $G$ can be stated as the requirement that there is no $x\in K$  such that   $(x,{a})\in \overline{\{x\in K^c\colon \Theta(x)\in K\}}$ for every $a \in \A$.\footnote{Here and throughout, $K^c$ and $\overline{K}$ denote complement and closure, respectively.} For a nonempty  subset $K \subset X$, we  let $\Ke$ denote the set $\{x\in K^c\colon \Theta(x)\in K\}$ and we refer to points of $\Ke$ as  points of exit  from $K$. The following results are presented in this note:  In Section \ref{sect: exit and G}  we show that if $\A$ is finite and $K\subset X$ is  a subshift, then $\overline{\Ke}\cap K=\emptyset$ if and only if $K$ is a subshift of finite type    (see Section \ref{sect: Notation} for the definition). From this result we obtain an alternative proof of Krieger's \cite{Kri06} theorem that a  subshift    with a  strictly positive $g$-function  ((\ref{def: invariance g}) holds and $g(x)>0$ for all $x\in K$)     must be of finite type if $\A$  is finite.  
In Section \ref{sect: strict g and strictly g} we consider the case when $\A$ is countably infinite. We   show  that property $G$ remains to be necessary and sufficient for a subshift to be invariant for  some $g\in \G$, but that  subshifts with   strictly positive  $g$-functions need not be of finite type. Here the condition that $\overline{\Ke}$ and $K$ do not intersect is  both necessary and sufficient.

%------- NEW SECTION

\section{Notation}\label{sect: Notation}
The notation in this paper follows \cite{Gun07}. The alphabet $\A$ is assumed   finite or countably infinite, 
%I let 
$X$ denotes  the set $\A^{\mathbb{Z}_-}$ of all sequences $x=(\dots{},x_{-1},x_0)$ of symbols from $\A$.  Let $\W$ denote the collection of all finite words over $\A$  and let $\W_n\subset \W, n\geq 1,$ be the  words of length $n$.  The word $w$   is said to appear in $x \in X$ if  $(x_{i},\dots{},x_{j})=w$ for some $i \leq j \leq 0$.  

The cylinders
$$C(w):=\{x\in X: x=(x',w)\text { for some }x' \in X\},\text{ $w\in \W$,  }$$
give a basis for the product topology (each copy of $\mathcal{A}$ is given the discrete topology) on $X$. We will use the fact that  cylinders are clopen (that is, closed and open), that  $X$ is  metrizable with 
\begin{equation}
\rho(x,x')=\begin{cases}
0 \mbox{ if $x=x'$},\\
2^{-l(x,x')}\text{ if $x\neq x'$, $l(x,x')=\min\{|j|:x_j \neq x_j'\}$},\\
\end{cases}
\end{equation}
and that $X$ is compact if $\A$ is finite. By a subshift of $X$, we mean  a closed nonempty  subset $K \subset X$ with $K=\Theta(K)$, where $\Theta(x)_j{:=}x_{j-1},j\in \Zm$. 
Subshifts are defined by forbidden words  (see  \cite[Chapter 1.2]{LM95}). 
 For  $\mathscr{F}\subset \W$, we let $X_\mathscr{F}$ be the  set of all $x\in X$ in which no word of $\mathscr{F}$ appears. A subshift  is    of  finite type if it can be written $X_\mathscr{F}$ for some   $\mathscr{F}$ containing finitely many words.  

%----- SECTION
 
\section{Exit points  and property $G$}\label{sect: exit and G}

As in \cite{Kri06}, we say that a subshift  $K\subset X$  has a $g$-function if $K$  is invariant for some   $g \in \mathcal{G}$. That is,   $K$  has a $g$-function if there exists a $g\in \mathcal{G}$ with the property that $g(x)=0$ for all $x\in \Ke:=\{x\in K^c\colon \Theta(x)\in K\}$. That $K$ has  a  \emph{strict} $g$-function means that $K$ is invariant for some $g\in \mathcal{G}$ with  $\{x\in X\colon g(x)=0\}=\overline{\Ke}$.  Finally, that $K$ has a \emph{strictly positive} $g$-function means that $K$ is invariant for  some $g\in \mathcal{G}$ such that  $g(x)>0$ for all $x\in K$. Invariance of   arbitrary  subsets of $X$ with respect to $g$-functions is defined in the same way.

Krieger \cite[pp. 306-307]{Kri06} uses the following notation (translated from the two-sided setting) to state his necessary and sufficient condition for a finite alphabet subshift  $K\subset X$ to have a  $g$-function: for $x \in X, a\in \A$, and  $w \in \W$, let 
\begin{align}
 \Gamma^-_\infty(a)&=\{x'\in K \colon (x',a) \in K\},\\
\omega^+_1(w)&=\bigcap_{x' \in \Gamma^-_\infty(a')}\{a' \in \A \colon (x',w,{a'}) \in K\}, \\
\triangle_1^+(x)&=\bigcup_{n\in \N}\omega_1^+(x_{[-n,0]}).
\end{align}
Property $G$ is the condition  that 
\begin{align}\label{def: property g}
\text{  }\triangle_1^+(x) \neq \emptyset \text{ for every $x \in X$}.  
\end{align}
In the present  notation, (\ref{def: property g}) says that
\begin{equation}
\text{there is no $x\in X$ with  $(x,{a}) \in \overline{\Ke}$ for every $a \in \A$. }
\label{eq: disjointness condition for K}
\end{equation}

\begin{remark}\label{rem: def g and G} If $K\subset  X$ is a subshift,  then  (\ref{eq: disjointness condition for K}) is equivalent to that there are  no $x\in K$ with $(x,{a}) \in \overline{\Ke}$ for every $a \in \A$.
\end{remark}
Conditions (\ref{def: property g}) and (\ref{eq: disjointness condition for K})  have  formulations for   general subsets  of $X$ with  $\A$ finite or countably   infinite. The proof of the following result shows that they are easily translated into each other.  

\begin{proposition}\label{prop: equivalence} For a subset  $K \subset X$,   (\ref{def: property g}) is satisfied  if and only if  (\ref{eq: disjointness condition for K}) is satisfied. 
\end{proposition} 
\begin{proof} For   $x\in X$ and ${a}\in \A$,  we have  $(x,{a}) \in \overline{\Ke}$ if and only if $a\notin \triangle_1^+(x)$. (If $(x,{a}) \in \overline{\Ke}$, then there is for every $n\in\N$ some $x' \in K$ with $x'_{[-n,0]}=x_{[-n,0]}$ and $(x',{a}) \in K^c$. This means that $a\notin \omega_1^+(x_{[-n,0]})$ for all  $n\in\N$, so that  $a \notin \triangle_1^+(x)$. Reversing the argument gives that $a\notin \triangle_1^+(x)$ implies $(x,{a}) \in \overline{\Ke}$.) So  $(x,{a}) \in \overline{\Ke}$ for all ${a}\in \A$ if and only if $\triangle_1^+(x)=\emptyset$. Thus  (\ref{def: property g}) is equivalent to that there is no $x\in X$ such  that $(x,{a}) \in \overline{\Ke}$ for every $a\in \A$. 
\end{proof}
 
Subshifts of finite type have property $G$. For a subshift of $\{0,1\}^{\mathbb{Z}_-}$  with property $G$ that is not of finite type, let  $K$ be the set of all $x\in  \{0,1\}^{\mathbb{Z}_-}$ such that between any two 1s there is an even number of 0s (cf. \cite[p. 6]{LM95}).  If $(x,{a}) \in \Ke$, then $a=1$. So there are no $x \in \{0,1\}^{\mathbb{Z}_-}$  for which $(x,0)$ and $(x,1)$ are both in $\overline{\Ke}$. 

Gundy \cite{Gun07} studies  Markov processes $\boldsymbol{x}_{t}, t=0,1,2,\dots$, on the binary sequence space $\{0,1\}^\Zm$.  Given $\boldsymbol{x}_{t}$, the process moves from $\boldsymbol{x}_{t}$ to   $(\boldsymbol{x}_{t},0)$ or  $(\boldsymbol{x}_{t},1)$, where the transition  probabilities  are  $g((\boldsymbol{x}_{t},0))$ and $g((\boldsymbol{x}_{t},1))$, respectively. For a subshift $K\subset \{0,1\}^\Zm$, points of $\Ke$ might be called  points of exit   from  $K$: if the process  leaves  $K$   at  time $t\geq 1$,   then we have $\boldsymbol{x}_t\in K^c$ and  $\Theta(\boldsymbol{x}_t) \in K$.  Paths from initial points in $K$ remain in $K$ with probability one if    $g(x)=0$   for all $x\in \Ke$.\footnote{Equivalently,  $K \subset \{0,1\}^{\mathbb{Z}_-}$ is invariant if $g(x)=1$  for all $x\in K$ with $(\dots,x_{-1},1-x_0) \in K^c$. In   \cite{Gun07}, such  points are called \emph{barrier points} of $K$. If $\A$ has more than two symbols however, then then we may have $g(x)=0$ for all $x\in \Ke$ and $g(x)<1$ for every $x\in X$.}   

If $K\subset X$ is $g$-invariant, then  $g(x)=0$ for all $x\in \overline{\Ke}$. So for $K$  to have a strictly positive $g$-function,  we must have $K \cap \overline{\Ke}=\emptyset$.  
 
\begin{proposition}\label{prop: properties of Ke}
 For $\A$  finite, the following are equivalent for a subshift  $K \subset X$: (i) $K$ is of finite type. (ii) $\Ke$ is closed. (iii) $K \cap \overline{\Ke}=\emptyset$.
\end{proposition}
 \begin{remark}\label{rem: Ke and SFT} As the proof shows,  (i) implies (ii) and (iii) also when is $\A$  is infinite. \end{remark}
\begin{proof} Write $K=X_\mathscr{F}$ for some collection of words $\mathscr{F}$. Assume without loss of generality that no word in $\mathscr{F}$ contains another word of $\mathscr{F}$ as a subword, i.e.,  if $w=(w_{-n},\dots,w_0) \in \mathscr{F}$, then    $(w_{-n+1},\dots,w_0)\notin \mathscr{F}$ and  $(w_{-n},\dots,w_{-1})\notin \mathscr{F}$.

\noindent (i) implies (ii): We can write
\begin{align}\label{eq: expression Ke} \Ke=\Theta^{-1}(K) \bigcap (\bigcup_{w \in \mathscr{F}}  C(w)).
\end{align}  
Here  $\Theta^{-1}(K)$ is closed since $K$ is closed and $\Theta$ is continuous.  Since $C(w)$ is closed for every $w \in \mathscr{F}$, $\cup_{w \in \mathscr{F}}  C(w)$ is closed if $\mathscr{F}$ is finite. From this and  (\ref{eq: expression Ke}) it follows that $\Ke$ is closed if  $\mathscr{F}$ is finite. 

\noindent (ii) implies (iii)  
since $\Ke\subset K^c$ by definition. 

\noindent (iii) implies (i): Assume that $K \cap \overline{\Ke}=\emptyset$ and suppose for contradiction that $\mathscr{F}$ is infinite. For  $i >1$,  choose a word $w=(w_{1},\dots{},w_{n_i}) \in \mathscr{F}$ of length $n_i \geq i$. By our assumption on $\mathscr{F}$,   $(w_{2},\dots{},w_{n_i})$ and $(w_{1},\dots{},w_{n_i-1})$ are both admissible (i.e., they appear in   points of $K$).  Take $x(i) \in K,i>1,$ such that $(w_{1},\dots{},w_{n_i-1})$ is the initial word in $x(i)$ of length $n_i-1$. Then  the initial word in $x'(i){:=}(x(i),w_{n_i})$ of length $n_i-1$ is admissible, but the initial word of length $n_i$ is forbidden. That is,   $x'(i) \in \Ke$ and $\rho(x'(i),K) \leq 2^{-n_i+1}$. Since $X$ is  compact,   $x'(i)$ has a convergent subsequence $x'(i_k),k\geq 1$, and since $\rho(x'(i_k),K) \to 0$,  we have $\lim_{k \to \infty}x'(i_k) \in K \cap \overline{\Ke}$. 
\end{proof}
 
%-------   SECTION 

\section{Strict $g$-functions and strictly positive $g$-functions}\label{sect: strict g and strictly g}
We begin by showing that property $G$ remains to be necessary and sufficient for a subset $K\subset X$ to have a $g$-function  if $\A$ is allowed to be  countably infinite.  

 \begin{theorem}\label{thm: strict g for c} For $\A$ finite or  countably infinite, the following are equivalent for a subset $K \subset X$: (i) $K$ has property $G$. (ii)  $K$ has a $g$-function. (iii)   
$K$ has a  strict $g$-function.
\end{theorem}
The finite alphabet part of this result is due to Krieger \cite{Kri06}.  Here  is a slight modification of the construction in \cite[p. 310]{Kri06}: for $x \in X$ and $a\in \A$, denote
\begin{align} 
x^{\ast,a}&{=}(\dots,x_{-2},x_{-1},a),\label{def: ast map points}\\ 
n(x)&=\begin{cases} \min\{n\in\N\colon C(x_{[-n,0]}) \subset (\overline{\Ke})^c\} \text{ if }x \in (\overline{\Ke})^c,\\
 \infty \text{ if }x \in \overline{\Ke},\\
 \end{cases}\\
\intertext{and let (here $1/\infty:=0$)  } g(x)&=
\frac{1/n(x)}{\sum_{j \in \A }1/n(x^{\ast,a})}.\label{def: Krieger's g}
\end{align}
Property $G$ assures that for every $x\in X$,   $n(x^{\ast,a})$ is finite for at least one ${a}\in \A$. Thus $g$ is well defined on all of $X$.  That   $\A$ is finite makes $\sum_{j \in \A }1/n(x^{\ast,a})$   finite for   every  $x\in X$, which means that  $\sum_{{a}\in\A}g(x^{\ast,a})=1$ for all  $x\in X$.  
Since  $g(x)=0$ if $x\in\Ke$,    $K$ is  $g$-invariant.

The   function (\ref{def: Krieger's g}) is still well defined if  $\A$ is infinite. But it will not generally be the case that (\ref{eq: sum is one}) holds and there is also the issue of continuity. The construction    below defines  $g(x)$ in terms of  $  \rho(x,\Ke)$.\footnote{This technique is used in \cite{Gun07} to  construct   $g$-functions for a class of  shifts of finite type.}  The $g$-function  is essentially a weighted sum of translates of $\rho(.,\Ke)$,  the partial sums of which  converge uniformly.    

\begin{proof}[Proof of Theorem \ref{thm: strict g for c}]
We only show that (i) implies (iii), the other implications being obvious. 
 
Fix a sequence $\lambda_a,a\in \A,$ of weights with $\lambda_a> 0$ and $\sum_{j \in \A}\lambda_a=1$. For $m>1$, let $\E_m$ be the collection of  $w \in  \W_m$   for which  $C(w)$ contains a point of $\overline{\Ke}$. Given $w=(w_{-m+1},\dots,w_{-1},w_0)\in \W_m$ and $a\in \A$, denote
\begin{align} 
w^{\ast,a}&{=}(w_{-m+1},\dots,w_{-1},a). \label{def: ast map words}
\end{align} 
That $K$ has property $G$ means that we can fix $m$ such that  if  $w\in \E_m$, then  $w^{\ast,b}\notin \E_m$ for some $b=b(w)\in \A$. Given $w \in \E_m$ and  $b=b(w) \in \A$, define $g$ on  
$$U(w){:=}\bigcup_{a\in \A}C(w^{\ast,a})$$
by setting 
\[
g(x)=\begin{cases}
\lambda_{x_0}\rho(x,\Ke) & \text{ if $x_0\neq b$, i.e., if $x\in \bigcup_{a\in \A\backslash\{b\}}C(w^{\ast,a})$},\\
1- \sum_{a \in \A\backslash\{b\}}\lambda_a\rho(x^{\ast,a},\Ke) &\text{ if $x_0=b$, i.e., if $x\in C(w^{\ast,b})$}.
 \end{cases}
\] 
If $w,w'\in \E_m$,  then we either have $U(w)=U(w')$ or $U(w)\cap U(w')=\emptyset$. Thus
\begin{align}\label{eq: Ue}
U_{\E_m}:=\bigcup_{w\in \E_m}U(w)
\end{align}
is a  disjoint union of sets $U(w), w \in \E_m$, which means that $g$ is well defined on $U_{\E_m}$. If $x$ is in $(U_{\E_m})^c=\cup_{w \in \W_m\backslash \E_m}U(w)$, then so is $ x^{\ast,a}$ for every ${a}\in \A$. So if we define  $g(x)=\lambda_{x_0}$ for $x\in(U_{\E_m})^c$, then we have $\sum_{{a}\in \A}g(x^{\ast,a})=1$ for all $x\in X$. 

It remains to show: (I) $g$ is continuous, (II) $g\geq 0$,  (III) $\{x\in X:g(x)=0\}=\overline{\Ke}$.

\noindent I: That $g$ is continuous on $C(w)$ for each $w$ in (\ref{eq: Ue}) follows from that $\rho(x,\Ke)$  is continuous and that  
the partial sums of $ \sum_{a \in \A\backslash\{b\}}\lambda_a\rho(x^{\ast,a},\Ke) $ converge uniformly. Since    $C(w)$ is open for each $w \in \E_m$, $g$ is continuous on $U(w)$. Since each  
 $U(w)$ is open, $g$ is continuous on  $U_{\E_m}$. Since  $(U_{\E_m})^c$ is  open, $g$ is continuous on $X$.

\noindent II: That $g\geq 0$ follows  from that $\rho \geq 0$,  $\lambda_a\geq 0$,  and $\sum_{a \in \A}\lambda_a=1$.

\noindent III:  If $x \in \overline{\Ke}$, then $x \in C(w)$ for some  $w \in \E_m$. Since $C(w^{\ast,b}) \cap \overline{\Ke}=\emptyset$, we must have $x\in \cup_{a\in \A\backslash\{b\}}C(w^{\ast,a})$,  
so that $g(x)=\lambda_{x_0}\rho(x,\Ke)=0$. Conversely, if   $g(x)=0$, then it cannot be the case that $g(x)=1- \sum_{a \in \A\backslash\{b\}}\lambda_a\rho(x^{\ast,a},\Ke)$, because $\lambda_b>0$ and $\rho \leq 1$. Thus  $g(x)=\lambda_{x_0}\rho(x,\Ke)$. Since  $\lambda_{x_0}>0$, we must have $\rho(x,\Ke)=0$ and therefore  $x\in \overline{\Ke}$.
\end{proof}

From   Theorem \ref{thm: strict g for c} we see that a subshift $K\subset X$ with property $G$ has a strictly positive $g$-function if  $K\cap \overline{\Ke}=\emptyset$. The latter condition   implies property $G$ and is therefore both necessary and sufficient:   

\begin{proposition}\label{thm: strictly g for c}  
For $\A$   finite or countably infinite, the following are equivalent for a subshift $K\subset X$: (i) $K$ has a strictly positive $g$-function. (ii) $K\cap \overline{\Ke}=\emptyset$.  (iii)  $\Ke$ is closed. 
\end{proposition}
\begin{proof} 
\noindent (i) implies (ii) since $g(x)=0$ for all $x\in  \overline{\Ke}$ if $K$ is $g$-invariant.  

\noindent (ii) implies (iii):  Suppose that  $\Ke$ is not closed. We show that  the boundary of $\Ke$ lies in $K$. Take $x\in (\Ke)^c$ on the boundary of $\Ke$ and a sequence  $x^n$ in $\Ke$ with $x^n \to x$.   That  $x^n$ is in $\Ke$  means that $\Theta(x^n)\in K$. Since $\Theta$ is continuous and $K$ is closed,  $\Theta(x^n) \to  \Theta(x)\in K$. Because $x$ was selected from the   complement of $\Ke$, we must have $x\in K$.

\noindent (iii) implies (i): Assume that  $\Ke$ is  closed. Property $G$ then  says that there are no $x\in X$ with $(x,a) \in \Ke$ for every $a\in \A$. Since $(x,a) \notin \Ke$ if $x\in K^c$, this amounts to the condition that there are no $x\in K$ with $(x,a) \in \Ke$ for every $a\in \A$, which  is another way of saying  that $K \subset \Theta(K)$.  Since $K= \Theta(K)$, we  can   conclude that $K$ has  property $G$.  By   Theorem \ref{thm: strict g for c} there exists a $g\in \G$ with $\{x\in X\colon g(x)=0\}=\overline{\Ke}$.   Because $\Ke$ is closed,  $\{x\in X\colon g(x)=0\}= \Ke\subset K$. 
\end{proof}
  
If $\A$ is countably infinite, then   it is still true that $K\cap \overline{\Ke}=\emptyset$  if $K$ is a subshift of finite type (see Remark \ref{rem: Ke and SFT}). To see that the reverse implication need not hold, let  $K$ be defined by proscribing all  but one of the  symbols of $\A$.  Then $K\cap \overline{\Ke}=\emptyset$, but $K$ is not of finite type. More generally, if $\F$ is infinite but the length of every  $w\in\F$ is less than $M$ for some $M<\infty$, then the distance between $K=X_\mathscr{F}$ and $\Ke$ is at least $2^{-M}$. (Ott et al \cite{OTW14} refer to such subshifts  as $M$-step shifts.) If no such $M$ exists there is a sequence $x^n$ in $\Ke$ with $\rho(x^n,K) \to 0$, but  since $X$ is not compact, the sequence may not  have a limit point. For example, if no two words of $\F$  have a symbol in common, then $K\cap \overline{\Ke}=\emptyset$ holds even if $\F$ contains arbitrarily long words.  
%\bibliographystyle{alpha}
%\bibliography{QMFreferences}
%\end{document}

 % \section*{Acknowledgement}  I am indebted to  Wolfgang Krieger for information about $g$-function presentations.

%%%%%%%%%%%%%%%%%%%%%%%%%%%%%%%%%%%%%%%%%%%%%%%%%%%%%%%%%%%%%%%%%%%
%%                                                               %%
%% Use the two commands below for producing your bibliography    %%
%% with bibtex, then comment again the commands and include the  %%
%% content of the .bbl file in this file below the commands.     %%
%%                                                               %%
%%%%%%%%%%%%%%%%%%%%%%%%%%%%%%%%%%%%%%%%%%%%%%%%%%%%%%%%%%%%%%%%%%%
%\bibliographystyle{alpha}
%\bibliographystyle{alpha}
%\bibliography{QMFreferences}
%\end{document}
% 

\end{document}